\documentclass[oneside,english]{amsart}
\usepackage[T1]{fontenc}
\usepackage[latin9]{inputenc}
\usepackage{geometry}
\usepackage{units}
\usepackage{amstext}
\usepackage{amsthm}
\usepackage{amssymb}

\makeatletter
\numberwithin{equation}{section}
\numberwithin{figure}{section}
\theoremstyle{plain}
\newtheorem{thm}{\protect\theoremname}
\theoremstyle{definition}
\newtheorem{defn}[thm]{\protect\definitionname}
\theoremstyle{plain}
\newtheorem{prop}[thm]{\protect\propositionname}
\theoremstyle{plain}
\newtheorem*{lem*}{\protect\lemmaname}

\usepackage{nicefrac}
\usepackage{babel}
\@ifundefined{definecolor}
 {\usepackage{color}}{}
\@ifundefined{definecolor}{\usepackage{color}}{}
\usepackage{graphicx}

\usepackage[mathscr]{euscript}
 \let\mathscr\relax
\usepackage[scr]{rsfso}

\usepackage{mathtools}

\makeatother

\usepackage{babel}
\providecommand{\definitionname}{Definition}
\providecommand{\lemmaname}{Lemma}
\providecommand{\propositionname}{Proposition}
\providecommand{\theoremname}{Theorem}

\begin{document}
\title{On The Composition Lemma}
\author{Edinah K. Gnang }
\begin{abstract}
We describe two distinct simple, short and self contained proofs of
the composition lemma.
\end{abstract}

\maketitle

\section{Introduction}

In \cite{gnang2020graceful} Gnang states a Composition Lemma which
yields as an immediate corollary a proof of the long standing Kötzig--Ringel--Rosa
conjecture. Unfortunately the first purported proof of this important
lemma as described in \cite{gnang2020graceful} is unnecessarily laborious
and incomplete. We describe here a much simpler rigorous proof and
a second proof of a slightly stronger version of the assertion of
this lemma. The new proof is based upon the polynomial method and
uses only elementary properties of the ring of symmetric polynomials.
Our aim in writing this note is to facilitate the application of the
composition lemma to other problems.

\section{Preliminaries}

For notational convenience let $\mathbb{Z}_{n}$ denote the subset
formed by the first $n$ consecutive non-negative integers i.e. $\mathbb{Z}_{n}=\left\{ 0,\cdots,n-1\right\} $.
For an arbitrary function
\[
f:\mathbb{Z}_{n}\rightarrow\mathbb{Z}_{n}
\]
\[
\text{we define}
\]
\[
\forall\,i\in\mathbb{Z}_{n},\;f^{(0)}\left(i\right)\,:=i\mbox{ and }\forall\,k\ge0,\;f\circ f^{\left(k\right)}=f^{\left(k+1\right)}=f^{\left(k\right)}\circ f.
\]
Such a function is said to lie in the transformation monoid $\mathbb{Z}_{n}^{\mathbb{Z}_{n}}$
and has a corresponding \emph{functional directed graph }denoted \emph{$G_{f}$}
whose vertex and directed edge sets are respectively
\[
V\left(G_{f}\right):=\mathbb{Z}_{n}\quad\text{ and }\quad E\left(G_{f}\right):=\left\{ \left(i,f\left(i\right)\right):\,i\in\mathbb{Z}_{n}\right\} .
\]
Finally, Aut$G_{f}$ denotes the automorphism group of $G_{f}$.

\subsection{Useful facts about polynomials}

Let 
\[
F\left(\mathbf{x}\right),\,G\left(\mathbf{x}\right)\in\mathbb{Q}\left[x_{0},\cdots,x_{n-1}\right],
\]
be multivariate polynomials which splits into irreducible factors
such that 
\[
F\left(\mathbf{x}\right)=\prod_{0\le i<m}\left(P_{i}\left(\mathbf{x}\right)\right)^{\alpha_{i}},\quad G\left(\mathbf{x}\right)=\prod_{0\le i<m}\left(P_{i}\left(\mathbf{x}\right)\right)^{\beta_{i}},
\]
where $\left\{ \alpha_{i},\beta_{i}\,:\,0\le i<m\right\} \subset\mathbb{Z}_{\ge0}$
and each factor $P_{i}\left(\mathbf{x}\right)$ is of the form 
\[
P_{i}\left(\mathbf{x}\right)=\sum_{j\in\mathbb{Z}_{n}}a_{ij}\,x_{j},\text{ where }\left\{ a_{ij}\,:\,\begin{array}{c}
0\le i<m\\
0\le j<n
\end{array}\right\} \subset\mathbb{Q}.
\]
Assume that $P_{i}\left(\mathbf{x}\right)$ has no common roots in
the field of fractions $\mathbb{Q}\left(x_{0},\cdots,x_{k-1},x_{k+1},\cdots,x_{n-1}\right)$,
with any other factor in $\left\{ P_{j}\left(\mathbf{x}\right)\,:\,0\le j\ne i<m\right\} $
for each $k\in\mathbb{Z}_{n}$, then 
\[
\text{LCM}\left(F\left(\mathbf{x}\right),\,G\left(\mathbf{x}\right)\right)\,:=\,\prod_{0\le i<m}\left(P_{i}\left(\mathbf{x}\right)\right)^{\max\left(\alpha_{i},\beta_{i}\right)}.
\]

\begin{defn}
By the quotient remainder theorem an arbitrary $H\left(\mathbf{x}\right)\in\mathbb{Q}\left[x_{0},\cdots,x_{n-1}\right]$
admits an expansion of the form
\[
H\left(\mathbf{x}\right)=\sum_{l\in\mathbb{Z}_{n}}q_{l}\left(\mathbf{x}\right)\,\left(x_{l}\right)^{\underline{n}}+\sum_{g\in\mathbb{Z}_{n}^{\mathbb{Z}_{n}}}H\left(g\left(0\right),\cdots,g\left(i\right),\cdots,g\left(n-1\right)\right)\,\prod_{k\in\mathbb{Z}_{n}}\left(\prod_{j_{k}\in\mathbb{Z}_{n}\backslash\left\{ g\left(k\right)\right\} }\left(\frac{x_{k}-j_{k}}{g\left(k\right)-j_{k}}\right)\right),
\]
where for notational convenience we adopt the falling factorial notation
\[
x^{\underline{n}}\,:=\prod_{i\in\mathbb{Z}_{n}}\left(x-i\right).
\]
 Incidentally, the canonical representative of the congruence class
\[
H\left(\mathbf{x}\right)\mod\left\{ \begin{array}{c}
\left(x_{i}\right)^{\underline{n}}\\
i\in\mathbb{Z}_{n}
\end{array}\right\} ,
\]
is defined as the unique polynomial of degree at most $\left(n-1\right)$
in each variable whose evaluations matches exactly evaluations of
$H\left(\mathbf{x}\right)$ over the integer lattice $\left(\mathbb{Z}_{n}\right)^{n}$.
The canonical representative is explicitly expressed as 
\begin{equation}
\sum_{g\in\mathbb{Z}_{n}^{\mathbb{Z}_{n}}}H\left(g\left(0\right),\cdots,g\left(i\right),\cdots,g\left(n-1\right)\right)\,\prod_{k\in\mathbb{Z}_{n}}\left(\prod_{j_{k}\in\mathbb{Z}_{n}\backslash\left\{ g\left(k\right)\right\} }\left(\frac{x_{k}-j_{k}}{g\left(k\right)-j_{k}}\right)\right).\label{Canonical representative}
\end{equation}
 
\end{defn}

The canonical representative of $H\left(\mathbf{x}\right)$ modulo
algebraic relations $\left\{ \left(x_{i}\right)^{\underline{n}}\equiv0:i\in\mathbb{Z}_{n}\right\} $,
is thus obtained via Lagrange interpolation over the integer lattice
$\left(\mathbb{Z}_{n}\right)^{n}$ as prescribed by Eq. (\ref{Canonical representative}).
Alternatively, the canonical representative is obtained as the final
remainder resulting from performing $n$ Euclidean divisions irrespective
of the order in which distinct divisors are successively taken from
the set $\left\{ \left(x_{i}\right)^{\underline{n}}\,:\,i\in\mathbb{Z}_{n}\right\} $.
The following proposition expresses via a determinantal construction
a necessary and sufficient condition for a given $f\in\mathbb{Z}_{n}^{\mathbb{Z}_{n}}$
to be such that
\[
n=\max_{\sigma\in\text{S}_{n}}\left|\left\{ \left|\sigma f\sigma^{(-1)}\left(i\right)-i\right|:i\in\mathbb{Z}_{n}\right\} \right|.
\]
 
\begin{prop}[Determinantal Certificate]
 For $f\in\mathbb{Z}_{n}^{\mathbb{Z}_{n}}$, $n=\underset{\sigma\in\text{S}_{n}}{\max}\left|\left\{ \left|\sigma f\sigma^{(-1)}\left(i\right)-i\right|:i\in\mathbb{Z}_{n}\right\} \right|$
if and only if
\[
0\not\equiv\text{LCM}\left\{ \prod_{0\le i<j<n}\left(x_{j}-x_{i}\right),\,\prod_{0\le i<j<n}\left(\left(x_{f\left(j\right)}-x_{j}\right)^{2}-\left(x_{f\left(i\right)}-x_{i}\right)^{2}\right)\right\} \mod\left\{ \begin{array}{c}
\left(x_{k}\right)^{\underline{n}}\\
k\in\mathbb{Z}_{n}
\end{array}\right\} .
\]
\end{prop}

\begin{proof}
The LCM in the claim of the proposition is well defined since both
polynomials
\[
\prod_{0\le i<j<n}\left(x_{j}-x_{i}\right)\:\text{ and }\:\prod_{0\le i<j<n}\left(\left(x_{f\left(j\right)}-x_{j}\right)^{2}-\left(x_{f\left(i\right)}-x_{i}\right)^{2}\right),
\]
 split into products of linear forms. Given that we are reducing modulo
algebraic relations
\[
\left(x_{k}\right)^{\underline{n}}\equiv0,\:\forall\,k\in\mathbb{Z}_{n},
\]
the canonical representative of the congruence class is completely
determined by evaluations of the dividend at lattice points taken
from $\left(\mathbb{Z}_{n}\right)^{n}$ as prescribed by Eq. (\ref{Canonical representative}).
This ensures a discrete set of roots for the canonical representative
of the congruence class. On the one hand, the LCM polynomial construction
vanishes when we assign to $\mathbf{x}$ a lattice point in $\left(\mathbb{Z}_{n}\right)^{n}$,
only if one of the irreducible multilinear factors vanishes at the
chosen evaluation point. On the other hand, one of the factors of
the polynomial construction vanishes at a lattice point only if either
two distinct vertex variables say $x_{i}$ and $x_{j}$ are assigned
the same integer label (we see this from the vertex Vandermonde determinant
factor) or alternatively if two distinct edges are assigned the same
subtractive induced edge label (we see this from the edge Vandermonde
determinant factor). The proof of sufficiency follows from the observation
that the only possible roots over $\left(\mathbb{Z}_{n}\right)^{n}$
to the multivariate polynomial
\[
\text{LCM}\left\{ \prod_{0<i<j<n}\left(x_{j}-x_{i}\right),\prod_{0<i<j<n}\left(\left(x_{f\left(j\right)}-x_{j}\right)^{2}-\left(x_{f\left(i\right)}-x_{i}\right)^{2}\right)\right\} \mod\left\{ \begin{array}{c}
\left(x_{k}\right)^{\underline{n}}\\
k\in\mathbb{Z}_{n}
\end{array}\right\} ,
\]
arise from vertex label assignments $\mathbf{x}\in\left(\mathbb{Z}_{n}\right)^{n}$
in which either distinct vertex variables are assigned the same label
or distinct edges are assigned the same subtractive induced edge label.
Consequently, the congruence identity
\[
0\equiv\text{LCM}\left\{ \prod_{0<i<j<n}\left(x_{j}-x_{i}\right),\prod_{0<i<j<n}\left(\left(x_{f\left(j\right)}-x_{j}\right)^{2}-\left(x_{f\left(i\right)}-x_{i}\right)^{2}\right)\right\} \mod\left\{ \begin{array}{c}
\left(x_{k}\right)^{\underline{n}}\\
k\in\mathbb{Z}_{n}
\end{array}\right\} ,
\]
implies that $n>\underset{\sigma\in\text{S}_{n}}{\max}\left|\left\{ \left|\sigma f\sigma^{(-1)}\left(i\right)-i\right|:i\in\mathbb{Z}_{n}\right\} \right|$.
Furthermore, the proof of necessity follows from the fact that every
non-vanishing assignment to the vertex variables entries of $\mathbf{x}$
in the polynomial
\[
\prod_{0\le i\ne j<n}\left(x_{j}-x_{i}\right)\left(\left(x_{f\left(j\right)}-x_{j}\right)^{2}-\left(x_{f\left(i\right)}-x_{i}\right)^{2}\right),
\]
evaluates to
\[
\prod_{0\le i\ne j<n}\left(j-i\right)\left(j^{2}-i^{2}\right),
\]
thus completing the proof.
\end{proof}

\begin{defn}
Let $P\left(\mathbf{x}\right)\in\mathbb{Q}\left[x_{0},\cdots,x_{n-1}\right]$,
we denote by Aut$\left(P\left(\mathbf{x}\right)\right)$ the stabilizer
subgroup of S$_{n}\subset\mathbb{Z}_{n}^{\mathbb{Z}_{n}}$ of $P\left(\mathbf{x}\right)$
with respect to permutation of the variable entries of $\mathbf{x}$.
\end{defn}

\begin{prop}[Stabilizer subgroup]
 For an arbitrary $f\in\mathbb{Z}_{n}^{\mathbb{Z}_{n}}$, let $P_{f}\left(\mathbf{x}\right)\in\mathbb{Q}\left[x_{0},\cdots,x_{n-1}\right]$
be defined such that
\[
P_{f}\left(\mathbf{x}\right)=\prod_{0\le i\ne j<n}\left(x_{j}-x_{i}\right)\left(\left(x_{f\left(j\right)}-x_{j}\right)^{2}-\left(x_{f\left(i\right)}-x_{i}\right)^{2}\right),
\]
then 
\[
\text{Aut}\left(P_{f}\left(\mathbf{x}\right)\right)=\text{Aut}\left(G_{f}\right).
\]
\end{prop}

\begin{proof}
Note that for all $\gamma\in$ S$_{n}$, we have 
\[
\prod_{0\le i\ne j<n}\left(x_{j}-x_{i}\right)\left(\left(x_{f\left(j\right)}-x_{j}\right)^{2}-\left(x_{f\left(i\right)}-x_{i}\right)^{2}\right)=
\]
\[
\prod_{0\le i\ne j<n}\left(x_{\gamma\left(j\right)}-x_{\gamma\left(i\right)}\right)\left(\left(x_{f\gamma\left(j\right)}-x_{\gamma\left(j\right)}\right)^{2}-\left(x_{f\gamma\left(i\right)}-x_{\gamma\left(i\right)}\right)^{2}\right).
\]
Consequently for all $\sigma\in$ Aut$\left(G_{f}\right)$ we have
\[
\prod_{0\le i\ne j<n}\left(x_{\sigma\left(j\right)}-x_{\sigma\left(i\right)}\right)\left(\left(x_{\sigma f\left(j\right)}-x_{\sigma\left(j\right)}\right)^{2}-\left(x_{\sigma f\left(i\right)}-x_{\sigma\left(i\right)}\right)^{2}\right)=
\]
\[
\prod_{0\le i\ne j<n}\left(x_{\sigma\sigma^{\left(-1\right)}\left(j\right)}-x_{\sigma\sigma^{\left(-1\right)}\left(i\right)}\right)\left(\left(x_{\sigma f\sigma^{\left(-1\right)}\left(j\right)}-x_{\sigma\sigma^{\left(-1\right)}\left(j\right)}\right)^{2}-\left(x_{\sigma f\sigma^{\left(-1\right)}\left(i\right)}-x_{\sigma\sigma^{\left(-1\right)}\left(i\right)}\right)^{2}\right)=P_{f}\left(\mathbf{x}\right).
\]
It also follows that for every permutation representative $\sigma$
of any coset of $\text{Aut}\left(G_{f}\right)$ other then $\text{Aut}\left(G_{f}\right)$
itself we have that
\[
\prod_{0\le i\ne j<n}\left(x_{\sigma\left(j\right)}-x_{\sigma\left(i\right)}\right)\left(\left(x_{\sigma f\left(j\right)}-x_{\sigma\left(j\right)}\right)^{2}-\left(x_{\sigma f\left(i\right)}-x_{\sigma\left(i\right)}\right)^{2}\right)=
\]
\[
\prod_{0\le i\ne j<n}\left(x_{\sigma\sigma^{\left(-1\right)}\left(j\right)}-x_{\sigma\sigma^{\left(-1\right)}\left(i\right)}\right)\left(\left(x_{\sigma f\sigma^{\left(-1\right)}\left(j\right)}-x_{\sigma\sigma^{\left(-1\right)}\left(j\right)}\right)^{2}-\left(x_{\sigma f\sigma^{\left(-1\right)}\left(i\right)}-x_{\sigma\sigma^{\left(-1\right)}\left(i\right)}\right)^{2}\right)\ne P_{f}\left(\mathbf{x}\right).
\]
We conclude that
\[
\text{Aut}\left(P_{f}\left(\mathbf{x}\right)\right)=\text{Aut}\left(G_{f}\right).
\]
\end{proof}

\section{The composition Lemma.}

The composition lemma was first proposed in \cite{gnang2020graceful}
as a critical step towards proving the long standing Kötzig--Ringel--Rosa
conjecture. The original statement \cite{gnang2020graceful} was restricted
to special members of the transformation monoid $\mathbb{Z}_{n}^{\mathbb{Z}_{n}}$.
Namely members which admit a unique fixed point which is attractive
over $\mathbb{Z}_{n}$. We state here the original version of the
composition lemma followed by a rigorous proof.
\begin{lem*}[Composition Lemma a]
 Let 
\[
f\in\left\{ h\in\mathbb{Z}_{n}^{\mathbb{Z}_{n}}:\begin{array}{c}
h^{\left(n-1\right)}\left(\mathbb{Z}_{n}\right)=\left\{ 0\right\} \\
h\left(i\right)\le i,\ \forall\:i\in\mathbb{Z}_{n}
\end{array}\right\} 
\]
 be such that Aut$\left(G_{f}\right)\subsetneq$ Aut$\left(G_{f^{\left(2\right)}}\right)$,
then
\[
\max_{\sigma\in\text{S}_{n}}\left|\left\{ \left|\sigma f^{\left(2\right)}\sigma^{(-1)}\left(i\right)-i\right|:i\in\mathbb{Z}_{n}\right\} \right|\le\max_{\sigma\in\text{S}_{n}}\left|\left\{ \left|\sigma f\sigma^{(-1)}\left(i\right)-i\right|:i\in\mathbb{Z}_{n}\right\} \right|.
\]
\end{lem*}
\begin{proof}
Note that for all
\[
\forall\:f\in\left\{ h\in\mathbb{Z}_{n}^{\mathbb{Z}_{n}}:\begin{array}{c}
h^{\left(n-1\right)}\left(\mathbb{Z}_{n}\right)=\left\{ 0\right\} \\
h\left(i\right)\le i,\ \forall\:i\in\mathbb{Z}_{n}
\end{array}\right\} ,
\]
the iterate $f^{\left(2^{\left\lfloor \ln_{2}\left(n\right)\right\rfloor }\right)}$
yields the identically zero function. Furthermore the existence of
\[
g\in\left\{ h\in\mathbb{Z}_{n}^{\mathbb{Z}_{n}}:\begin{array}{c}
h^{\left(n-1\right)}\left(\mathbb{Z}_{n}\right)=\left\{ 0\right\} \\
h\left(i\right)\le i,\ \forall\:i\in\mathbb{Z}_{n}
\end{array}\right\} \text{ such that }n>\max_{\sigma\in\text{S}_{n}}\left|\left\{ \left|\sigma g\sigma^{(-1)}\left(i\right)-i\right|:i\in\mathbb{Z}_{n}\right\} \right|,
\]
implies the existence of some function
\[
f\in\left\{ g^{\left(2^{\kappa}\right)}\,:\,0\le\kappa<\left\lceil \ln_{2}\left(n-1\right)\right\rceil \right\} \subset\left\{ h\in\mathbb{Z}_{n}^{\mathbb{Z}_{n}}\,:\,\begin{array}{c}
h^{\left(n-1\right)}\left(\mathbb{Z}_{n}\right)=\left\{ 0\right\} \\
h\left(i\right)\le i,\ \forall\:i\in\mathbb{Z}_{n}
\end{array}\right\} ,
\]
for which 
\[
\max_{\sigma\in\text{S}_{n}}\left|\left\{ \left|\sigma f^{\left(2\right)}\sigma^{(-1)}\left(i\right)-i\right|:i\in\mathbb{Z}_{n}\right\} \right|>\max_{\sigma\in\text{S}_{n}}\left|\left\{ \left|\sigma f\sigma^{(-1)}\left(i\right)-i\right|:i\in\mathbb{Z}_{n}\right\} \right|.
\]
Consequently if the purported claim of the composition lemma holds,
there can be no member $g$ in the semigroup for which 
\[
n>\max_{\sigma\in\text{S}_{n}}\left|\left\{ \left|\sigma g\sigma^{(-1)}\left(i\right)-i\right|:i\in\mathbb{Z}_{n}\right\} \right|.
\]
Consequently, to prove the desired claim it suffices to show that
\[
\forall\:f\in\left\{ h\in\mathbb{Z}_{n}^{\mathbb{Z}_{n}}:\begin{array}{c}
h^{\left(n-1\right)}\left(\mathbb{Z}_{n}\right)=\left\{ 0\right\} \\
h\left(i\right)\le i,\ \forall\:i\in\mathbb{Z}_{n}
\end{array}\right\} \,\text{ subject to Aut}\left(G_{f}\right)\subsetneq\text{Aut}\left(G_{f^{\left(2\right)}}\right)
\]
\[
n=\max_{\sigma\in\text{S}_{n}}\left|\left\{ \left|\sigma f^{\left(2\right)}\sigma^{(-1)}\left(i\right)-i\right|\,:\,i\in\mathbb{Z}_{n}\right\} \right|\implies n=\max_{\sigma\in\text{S}_{n}}\left|\left\{ \left|\sigma f\sigma^{(-1)}\left(i\right)-i\right|\,:\,i\in\mathbb{Z}_{n}\right\} \right|.
\]
Recall that $\mathbf{A}_{G_{f}}\in\left\{ 0,1\right\} ^{n\times n}$
denotes the adjacency matrix of $G_{f}$ whose entries are such that
\[
\mathbf{A}_{G_{f}}\left[i,j\right]=\begin{cases}
\begin{array}{cc}
1 & \text{ if }j=f\left(i\right)\\
0 & \text{otherwise}
\end{array}, & \forall\:0\le i,j<n\end{cases}.
\]
We express a variant of the determinantal certificate construction
as
\[
\prod_{\begin{array}{c}
t\in\left\{ 0,1\right\} \\
0\le i<j<n
\end{array}}\left[\left(\mathbf{I}_{n}\left[j,:\right]\left(\mathbf{A}_{G_{f}}-\mathbf{I}_{n}\right)^{t}\mathbf{x}\right)^{2}-\left(\mathbf{I}_{n}\left[i,:\right]\left(\mathbf{A}_{G_{f}}-\mathbf{I}_{n}\right)^{t}\mathbf{x}\right)^{2}\right]^{2}\text{ mod}\left\{ \begin{array}{c}
\left(x_{k}\right)^{\underline{n}}\\
k\in\mathbb{Z}_{n}
\end{array}\right\} .
\]
For notational convenience let 
\[
V\left(\mathbf{x}\right)=\prod_{0\le i<j<n}\left(x_{j}^{2}-x_{i}^{2}\right)^{2}.
\]
To an arbitrary $g\in\mathbb{Z}_{n}^{\mathbb{Z}_{n}}$, we associated
the polynomial 
\[
P_{g}:=\prod_{t\in\left\{ 0,1\right\} }V\left(\left(\mathbf{A}_{G_{g}}-\mathbf{I}_{n}\right)^{t}\mathbf{x}\right)=V\left(\mathbf{x}\right)V\left(\left(\mathbf{A}_{G_{g}}-\mathbf{I}_{n}\right)\mathbf{x}\right).
\]
The polynomial $P_{f^{\left(2\right)}}$ is thus given by 
\[
P_{f^{\left(2\right)}}=\prod_{t\in\left\{ 0,1\right\} }V\left(\left(\mathbf{A}_{G_{f^{\left(2\right)}}}-\mathbf{I}_{n}\right)^{t}\mathbf{x}\right)=V\left(\mathbf{x}\right)V\left(\left(\mathbf{A}_{G_{f^{\left(2\right)}}}-\mathbf{I}_{n}\right)\mathbf{x}\right).
\]
By the quotient remainder theorem $P_{f^{\left(2\right)}}$ admits
an expansion of the form
\[
P_{f^{\left(2\right)}}=V\left(\mathbf{x}\right)\left(\sum_{l\in\mathbb{Z}_{n}}\left(x_{l}\right)^{\underline{n}}q_{l}\left(\mathbf{x}\right)+\sum_{h\in\mathbb{Z}_{n}^{\mathbb{Z}_{n}}}V\left(\left(\mathbf{A}_{G_{f^{\left(2\right)}}}-\mathbf{I}_{n}\right)\mathbf{A}_{G_{h}}\mathbf{v}\right)\prod_{\begin{array}{c}
k\in\mathbb{Z}_{n}\\
j_{k}\in\mathbb{Z}_{n}\backslash\left\{ h\left(k\right)\right\} 
\end{array}}\left(\frac{\mathbf{I}\left[k,:\right]\mathbf{x}-j_{k}}{h\left(k\right)-j_{k}}\right)\right)
\]
where 
\[
\mathbf{v}\left[u\right]=u,\quad\forall\:u\in\mathbb{Z}_{n}.
\]
We then apply to the polynomial $P_{f^{\left(2\right)}}$ the map
prescribed by
\begin{equation}
\left(\mathbf{A}_{G_{f^{\left(2\right)}}}-\mathbf{I}_{n}\right)^{t}\mathbf{x}\mapsto\left(\mathbf{A}_{G_{f}}+\mathbf{I}_{n}\right)^{-t}\left(\mathbf{A}_{G_{f^{\left(2\right)}}}-\mathbf{I}_{n}\right)^{t}\mathbf{x},\ \forall\:t\in\left\{ 0,1\right\} .\label{Invertible linear map}
\end{equation}
More precisely the first factor $V\left(\mathbf{x}\right)$ of $P_{f^{\left(2\right)}}$
is left unchanged by the proposed map where as the map effects in
the second factor $V\left(\left(\mathbf{A}_{G_{f^{\left(2\right)}}}-\mathbf{I}_{n}\right)\mathbf{x}\right)$
of $P_{f^{\left(2\right)}}$ the linear transformation 
\[
\mathbf{x}\mapsto\left(\mathbf{A}_{G_{f}}+\mathbf{I}_{n}\right)^{-1}\mathbf{x}.
\]
The said transformation is well defined since
\[
2=\det\left(\mathbf{A}_{G_{f}}+\mathbf{I}_{n}\right),\quad\forall\:f\in\left\{ h\in\mathbb{Z}_{n}^{\mathbb{Z}_{n}}:\begin{array}{c}
h^{\left(n-1\right)}\left(\mathbb{Z}_{n}\right)=\left\{ 0\right\} \\
h\left(i\right)\le i,\ \forall\:i\in\mathbb{Z}_{n}
\end{array}\right\} .
\]
Furthermore the transformation maps $P_{f^{\left(2\right)}}$ to $P_{f}$
since
\[
\left(\mathbf{A}_{G_{f}}-\mathbf{I}_{n}\right)\left(\mathbf{A}_{G_{f}}+\mathbf{I}_{n}\right)=\left(\mathbf{A}_{G_{f^{\left(2\right)}}}-\mathbf{I}_{n}\right)=\left(\mathbf{A}_{G_{f}}+\mathbf{I}_{n}\right)\left(\mathbf{A}_{G_{f}}-\mathbf{I}_{n}\right),\quad\forall\,f\in\mathbb{Z}_{n}^{\mathbb{Z}_{n}}.
\]
\[
\implies P_{f}=V\left(\mathbf{x}\right)\sum_{l\in\mathbb{Z}_{n}}\left(\mathbf{I}\left[l,:\right]\left(\mathbf{A}_{G_{f}}+\mathbf{I}_{n}\right)^{-1}\mathbf{x}\right)^{\underline{n}}\,q_{l}\left(\left(\mathbf{A}_{G_{f}}+\mathbf{I}_{n}\right)^{-1}\mathbf{x}\right)+
\]
\[
V\left(\mathbf{x}\right)\sum_{h\in\mathbb{Z}_{n}^{\mathbb{Z}_{n}}}V\left(\left(\mathbf{A}_{G_{f^{\left(2\right)}}}-\mathbf{I}_{n}\right)\mathbf{A}_{G_{h}}\mathbf{v}\right)\prod_{\begin{array}{c}
k\in\mathbb{Z}_{n}\\
j_{k}\in\mathbb{Z}_{n}\backslash\left\{ h\left(k\right)\right\} 
\end{array}}\left(\frac{\mathbf{I}\left[k,:\right]\left(\mathbf{A}_{G_{f}}+\mathbf{I}_{n}\right)^{-1}\mathbf{x}-j_{k}}{h\left(k\right)-j_{k}}\right)
\]
The premise that $f$ lies in the semigroup described by the set 
\[
\left\{ h\in\mathbb{Z}_{n}^{\mathbb{Z}_{n}}:\begin{array}{c}
h^{\left(n-1\right)}\left(\mathbb{Z}_{n}\right)=\left\{ 0\right\} \\
h\left(i\right)\le i,\ \forall\:i\in\mathbb{Z}_{n}
\end{array}\right\} 
\]
implies that
\[
V\left(\left(\mathbf{A}_{G_{f^{\left(2\right)}}}-\mathbf{I}_{n}\right)\left(\mathbf{A}_{G_{f}}+\mathbf{I}_{n}\right)^{-1}\mathbf{x}\right)=\left(\sum_{l\in\mathbb{Z}_{n}}\left(\mathbf{I}\left[l,:\right]\left(\mathbf{A}_{G_{f}}+\mathbf{I}_{n}\right)^{-1}\mathbf{x}\right)^{\underline{n}}\,q_{l}\left(\left(\mathbf{A}_{G_{f}}+\mathbf{I}_{n}\right)^{-1}\mathbf{x}\right)\right)+
\]
 
\[
\sum_{h\in\mathbb{Z}_{n}^{\mathbb{Z}_{n}}}V\left(\left(\mathbf{A}_{G_{f^{\left(2\right)}}}-\mathbf{I}_{n}\right)\mathbf{A}_{G_{h}}\mathbf{v}\right)\prod_{j_{0}\in\mathbb{Z}_{n}\backslash\left\{ h\left(0\right)\right\} }\left(2^{-1}\frac{x_{0}-j_{0}}{h\left(0\right)-j_{0}}-\frac{2^{-1}j_{0}}{h\left(0\right)-j_{0}}\right)\times
\]
\[
\prod_{\begin{array}{c}
k\in\mathbb{Z}_{n}\backslash\left\{ 0\right\} \\
j_{k}\in\mathbb{Z}_{n}\backslash\left\{ h\left(k\right)\right\} 
\end{array}}\left(\frac{x_{k}-j_{k}}{h\left(k\right)-j_{k}}+\frac{\underset{0\le\tau<k}{\sum}\left(\mathbf{A}_{G_{f}}+\mathbf{I}_{n}\right)^{-1}\left[k,\tau\right]x_{\tau}}{h\left(k\right)-j_{k}}\right)
\]
For notational convenience let 
\[
\lambda_{k}:=\left(\mathbf{A}_{G_{f}}+\mathbf{I}_{n}\right)^{-1}\left[k,k\right]=\begin{cases}
\begin{array}{cc}
2^{-1} & \text{ if }k=0\\
1 & \text{otherwise}
\end{array} & \forall\,k\in\mathbb{Z}_{n}\end{cases}.
\]
Invoking the other binomial identity
\[
\prod_{i\in I}(a_{i}+b_{i})=\prod_{i\in I}a_{i}+\sum_{\substack{(k_{i}:i\in I)\in\{0,1\}^{n}\\
\underset{i\in I}{\prod}k_{i}=0
}
}\left(\prod_{i\in I}a_{i}^{k_{i}}b_{i}^{1-k_{i}}\right),
\]
yields
\[
V\left(\left(\mathbf{A}_{G_{f^{\left(2\right)}}}-\mathbf{I}_{n}\right)\left(\mathbf{A}_{G_{f}}+\mathbf{I}_{n}\right)^{-1}\mathbf{x}\right)=\left(\sum_{l\in\mathbb{Z}_{n}}\left(\mathbf{I}\left[l,:\right]\left(\mathbf{A}_{G_{f}}+\mathbf{I}_{n}\right)^{-1}\mathbf{x}\right)^{\underline{n}}\,q_{l}\left(\left(\mathbf{A}_{G_{f}}+\mathbf{I}_{n}\right)^{-1}\mathbf{x}\right)\right)+
\]
 
\[
\left(\sum_{h\in\mathbb{Z}_{n}^{\mathbb{Z}_{n}}}\frac{V\left(\left(\mathbf{A}_{G_{f^{\left(2\right)}}}-\mathbf{I}_{n}\right)\mathbf{A}_{G_{h}}\mathbf{v}\right)}{2^{n-1}}\prod_{\begin{array}{c}
k\in\mathbb{Z}_{n}\\
j_{k}\in\mathbb{Z}_{n}\backslash\left\{ h\left(k\right)\right\} 
\end{array}}\left(\frac{x_{k}-j_{k}}{h\left(k\right)-j_{k}}\right)\right)+
\]
\[
\sum_{h\in\mathbb{Z}_{n}^{\mathbb{Z}_{n}}}V\left(\left(\mathbf{A}_{G_{f^{\left(2\right)}}}-\mathbf{I}_{n}\right)\mathbf{A}_{G_{h}}\mathbf{v}\right)\sum_{\begin{array}{c}
a_{k,j_{k}}\in\left\{ 0,1\right\} \\
0=\prod_{k,j_{k}}a_{k,j_{k}}
\end{array}}\prod_{\begin{array}{c}
k\in\mathbb{Z}_{n}\\
j_{k}\in\mathbb{Z}_{n}\backslash\left\{ h\left(k\right)\right\} 
\end{array}}\left(\lambda_{k}\frac{x_{k}-j_{k}}{h\left(k\right)-j_{k}}\right)^{a_{k,j_{k}}}\times
\]
\begin{equation}
\left(\frac{\underset{0\le\tau<k}{\sum}\left(\mathbf{A}_{G_{f}}+\mathbf{I}_{n}\right)^{-1}\left[k,\tau\right]x_{\tau}+\left(\lambda_{k}-1\right)j_{k}}{h\left(k\right)-j_{k}}\right)^{1-a_{k,j_{k}}}.\label{Semi_expansion}
\end{equation}
The product of $V\left(\mathbf{x}\right)$ with the polynomial summand
\[
\sum_{h\in\mathbb{Z}_{n}^{\mathbb{Z}_{n}}}\frac{V\left(\left(\mathbf{A}_{G_{f^{\left(2\right)}}}-\mathbf{I}_{n}\right)\mathbf{A}_{G_{h}}\mathbf{v}\right)}{2^{n-1}}\prod_{\begin{array}{c}
k\in\mathbb{Z}_{n}\\
j_{k}\in\mathbb{Z}_{n}\backslash\left\{ h\left(k\right)\right\} 
\end{array}}\left(\frac{x_{k}-j_{k}}{h\left(k\right)-j_{k}}\right)
\]
taken from the semi expanded expression of $V\left(\left(\mathbf{A}_{G_{f^{\left(2\right)}}}-\mathbf{I}_{n}\right)\left(\mathbf{A}_{G_{f}}+\mathbf{I}_{n}\right)^{-1}\mathbf{x}\right)$
described in Eq (\ref{Semi_expansion}) is congruent to 
\[
\frac{P_{f^{\left(2\right)}}\left(\mathbf{x}\right)}{2^{n-1}}\text{ mod}\left\{ \left(x_{k}\right)^{\underline{n}}\,:\,k\in\mathbb{Z}_{n}\right\} .
\]
Furthermore
\[
n=\max_{\sigma\in\text{S}_{n}}\left|\left\{ \left|\sigma f^{\left(2\right)}\sigma^{(-1)}\left(i\right)-i\right|\,:\,i\in\mathbb{Z}_{n}\right\} \right|\implies P_{f^{\left(2\right)}}\left(\mathbf{x}\right)\not\equiv0\mod\left\{ \begin{array}{c}
\left(x_{k}\right)^{\underline{n}}\\
k\in\mathbb{Z}_{n}
\end{array}\right\} .
\]
 Recall that the automorphism group of $P_{f^{\left(2\right)}}\left(\mathbf{x}\right)$
is Aut$\left(G_{f^{\left(2\right)}}\right)$. Whereas the remaining
summand in the semi-expanded expression of $V\left(\left(\mathbf{A}_{G_{f^{\left(2\right)}}}-\mathbf{I}_{n}\right)\left(\mathbf{A}_{G_{f}}+\mathbf{I}_{n}\right)^{-1}\mathbf{x}\right)$
described in Eq (\ref{Semi_expansion}) has symmetry group Aut$\left(G_{f}\right)$.
\[
\implies\prod_{t\in\left\{ 0,1\right\} }V\left(\left(\mathbf{A}_{G_{f}}+\mathbf{I}_{n}\right)^{-t}\left(\mathbf{A}_{G_{f^{\left(2\right)}}}-\mathbf{I}_{n}\right)^{t}\mathbf{x}\right)\not\equiv0\mod\left\{ \begin{array}{c}
\left(x_{k}\right)^{\underline{n}}\\
k\in\mathbb{Z}_{n}
\end{array}\right\} 
\]
\[
\implies P_{f}\left(\mathbf{x}\right)\not\equiv0\mod\left\{ \begin{array}{c}
\left(x_{k}\right)^{\underline{n}}\\
k\in\mathbb{Z}_{n}
\end{array}\right\} .
\]
Hence 
\[
n=\max_{\sigma\in\text{S}_{n}}\left|\left\{ \left|\sigma f^{\left(2\right)}\sigma^{(-1)}\left(i\right)-i\right|\,:\,i\in\mathbb{Z}_{n}\right\} \right|\implies n=\max_{\sigma\in\text{S}_{n}}\left|\left\{ \left|\sigma f\sigma^{(-1)}\left(i\right)-i\right|\,:\,i\in\mathbb{Z}_{n}\right\} \right|.
\]
From which we conclude in turn that
\[
\max_{\sigma\in\text{S}_{n}}\left|\left\{ \left|\sigma f^{\left(2\right)}\sigma^{(-1)}\left(i\right)-i\right|:i\in\mathbb{Z}_{n}\right\} \right|\le\max_{\sigma\in\text{S}_{n}}\left|\left\{ \left|\sigma f\sigma^{(-1)}\left(i\right)-i\right|:i\in\mathbb{Z}_{n}\right\} \right|.
\]
as claimed.
\end{proof}
Note that the premise Aut$\left(G_{f}\right)\subsetneq$ Aut$\left(G_{f^{\left(2\right)}}\right)$
incurs no loss of generality when $n>2$. For we see that
\[
\forall\:f\in\left\{ h\in\mathbb{Z}_{n}^{\mathbb{Z}_{n}}:\begin{array}{c}
h^{\left(n-1\right)}\left(\mathbb{Z}_{n}\right)=\left\{ 0\right\} \\
h\left(i\right)\le i,\ \forall\:i\in\mathbb{Z}_{n}
\end{array}\right\} ,
\]
subject to Aut$\left(G_{f}\right)=$ Aut$\left(G_{f^{\left(2\right)}}\right)$
there exist 
\[
g\in\left\{ h\in\mathbb{Z}_{n}^{\mathbb{Z}_{n}}:\begin{array}{c}
h^{\left(n-1\right)}\left(\mathbb{Z}_{n}\right)=\left\{ 0\right\} \\
h\left(i\right)\le i,\ \forall\:i\in\mathbb{Z}_{n}
\end{array}\right\} ,
\]
whose graph $G_{g}$ differs from $G_{f}$ by a fixed point swap and
for which Aut$\left(G_{g}\right)\subsetneq$ Aut$\left(G_{g^{\left(2\right)}}\right)$.
For instance we devise such a graph $G_{g}$ systematically from $G_{f}$
by relocating the fixed point to a vertex at edge distance exactly
$2$ from a leaf node. Crucially, the resulting graphs $G_{g}$ differs
from $G_{f}$ by a fixed point swap and
\[
\left|\text{GrL}\left(G_{f}\right)\right|=\left|\text{GrL}\left(G_{g}\right)\right|.
\]

In contrast to the original statement of the composition lemma we
propose a variant applicable to all members the monoid $\mathbb{Z}_{n}^{\mathbb{Z}_{n}}$.
We now state and prove a completely different proof our stronger variant
of the composition lemma
\begin{lem*}[Composition Lemma b]
 Let $f\in\mathbb{Z}_{n}^{\mathbb{Z}_{n}}$ be such that Aut$\left(G_{f}\right)\subsetneq$
Aut$\left(G_{f^{\left(2\right)}}\right)$, then
\[
\max_{\sigma\in\text{S}_{n}}\left|\left\{ \left|\sigma f^{\left(2\right)}\sigma^{(-1)}\left(i\right)-i\right|:i\in\mathbb{Z}_{n}\right\} \right|\le\max_{\sigma\in\text{S}_{n}}\left|\left\{ \left|\sigma f\sigma^{(-1)}\left(i\right)-i\right|:i\in\mathbb{Z}_{n}\right\} \right|.
\]
\end{lem*}
\begin{proof}
Let 
\[
m=\max_{\sigma\in\text{S}_{n}}\left|\left\{ \left|\sigma f^{\left(2\right)}\sigma^{(-1)}\left(i\right)-i\right|:i\in\mathbb{Z}_{n}\right\} \right|,
\]
where $1<m\le n$. By Prop. (2), the desired claim is equivalent to
the assertion that
\[
0\not\equiv\sum_{\begin{array}{c}
S\subset\mathbb{Z}_{n}\\
\left|S\right|=m
\end{array}}\text{LCM}\left\{ \prod_{0\le i<j<n}\left(x_{j}-x_{i}\right)^{2},\prod_{\begin{array}{c}
i<j\\
\left(i,j\right)\in S\times S
\end{array}}\left(\left(x_{f^{\left(2\right)}\left(j\right)}-x_{j}\right)^{2}-\left(x_{f^{\left(2\right)}\left(i\right)}-x_{i}\right)^{2}\right)^{2}\right\} \text{mod}\left\{ \begin{array}{c}
\left(x_{k}\right)^{\underline{n}}\\
k\in\mathbb{Z}_{n}
\end{array}\right\} ,
\]
implies that 
\[
0\not\equiv\sum_{\begin{array}{c}
S\subset\mathbb{Z}_{n}\\
\left|S\right|=m
\end{array}}\text{LCM}\left\{ \prod_{0\le i<j<n}\left(x_{j}-x_{i}\right)^{2},\prod_{\begin{array}{c}
i<j\\
\left(i,j\right)\in S\times S
\end{array}}\left(\left(x_{f\left(j\right)}-x_{j}\right)^{2}-\left(x_{f\left(i\right)}-x_{i}\right)^{2}\right)^{2}\right\} \text{mod}\left\{ \begin{array}{c}
\left(x_{k}\right)^{\underline{n}}\\
k\in\mathbb{Z}_{n}
\end{array}\right\} .
\]
For simplicity and clarity we prove the claim when $n=m$. The argument
is the same in the setting where $1<m<n$ but considerably more burdensome
in notation. In other words we restrict the proof to the setting where
$m=n$ to avoid oppressive details. In particular we show that 
\[
0\not\equiv\text{LCM}\left\{ \prod_{0\le i<j<n}\left(x_{j}-x_{i}\right),\prod_{0\le i<j<n}\left(\left(x_{f^{\left(2\right)}\left(j\right)}-x_{j}\right)^{2}-\left(x_{f^{\left(2\right)}\left(i\right)}-x_{i}\right)^{2}\right)\right\} \text{mod}\left\{ \begin{array}{c}
\left(x_{k}\right)^{\underline{n}}\\
k\in\mathbb{Z}_{n}
\end{array}\right\} ,
\]
implies that 
\[
0\not\equiv\text{LCM}\left\{ \prod_{0\le i<j<n}\left(x_{j}-x_{i}\right),\prod_{0\le i<j<n}\left(\left(x_{f\left(j\right)}-x_{j}\right)^{2}-\left(x_{f\left(i\right)}-x_{i}\right)^{2}\right)\right\} \text{mod}\left\{ \begin{array}{c}
\left(x_{k}\right)^{\underline{n}}\\
k\in\mathbb{Z}_{n}
\end{array}\right\} .
\]
For this purpose, we associate with an arbitrary $g\in\mathbb{Z}_{n}^{\mathbb{Z}_{n}}$,
a multivariate polynomial $P_{g}\left(\mathbf{x}\right)\in\mathbb{Q}\left[x_{0},\cdots,x_{n-1}\right]$
given by 
\[
P_{g}\left(\mathbf{x}\right)=\prod_{0\le i\ne j<n}\left(x_{j}-x_{i}\right)\left(\left(x_{g\left(j\right)}-x_{j}\right)^{2}-\left(x_{g\left(i\right)}-x_{i}\right)^{2}\right).
\]
We prove the desired claim by contradiction. Assume that
\[
n=\max_{\sigma\in\text{S}_{n}}\left|\left\{ \left|\sigma f^{\left(2\right)}\sigma^{(-1)}\left(i\right)-i\right|:i\in\mathbb{Z}_{n}\right\} \right|\;\text{ and }\;\text{Aut}\left(G_{f}\right)\subsetneq\text{Aut}\left(G_{f^{\left(2\right)}}\right).
\]
Let 
\[
P_{f}\left(\mathbf{x}\right)=\prod_{0\le i\ne j<n}\left(x_{j}-x_{i}\right)\left(\left(x_{f\left(j\right)}-x_{j}\right)^{2}-\left(x_{f\left(i\right)}-x_{i}\right)^{2}\right).
\]
By telescoping and rearranging terms in the each factor we write
\[
\begin{array}{ccc}
P_{f}\left(\mathbf{x}\right) & = & \underset{0\le i\ne j<n}{\prod}\left(x_{j}-x_{i}\right)\left(\left(x_{f\left(j\right)}-x_{f^{\left(2\right)}\left(j\right)}+x_{f^{\left(2\right)}\left(j\right)}-x_{j}\right)^{2}-\left(x_{f\left(i\right)}-x_{f^{\left(2\right)}\left(i\right)}+x_{f^{\left(2\right)}\left(i\right)}-x_{i}\right)^{2}\right),\\
\\
\\
 & = & \underset{0\le i\ne j<n}{\prod}\left(x_{j}-x_{i}\right)\left(\left(x_{f\left(j\right)}-x_{f^{\left(2\right)}\left(j\right)}\right)^{2}+\left(x_{f^{\left(2\right)}\left(j\right)}-x_{j}\right)^{2}+2\left(x_{f\left(j\right)}-x_{f^{\left(2\right)}\left(j\right)}\right)\left(x_{f^{\left(2\right)}\left(j\right)}-x_{j}\right)+\right.\\
 &  & \left.\left(-1\right)\left(x_{f\left(i\right)}-x_{f^{\left(2\right)}\left(i\right)}\right)^{2}-\left(x_{f^{\left(2\right)}\left(i\right)}-x_{i}\right)^{2}-2\left(x_{f\left(i\right)}-x_{f^{\left(2\right)}\left(i\right)}\right)\left(x_{f^{\left(2\right)}\left(i\right)}-x_{i}\right)\right),\\
\\
\\
P_{f}\left(\mathbf{x}\right) & = & P_{f^{\left(2\right)}}\left(\mathbf{x}\right)+\underset{\begin{array}{c}
k_{ij}\in\left\{ 0,1\right\} \\
0=\underset{0\le i\ne j<n}{\prod}k_{ij}
\end{array}}{\sum}\underset{0\le i\ne j<n}{\prod}\left(x_{j}-x_{i}\right)\left(\left(x_{f^{\left(2\right)}\left(j\right)}-x_{j}\right)^{2}-\left(x_{f^{\left(2\right)}\left(i\right)}-x_{i}\right)^{2}\right)^{k_{ij}}\left(\left(x_{f\left(j\right)}-x_{f^{\left(2\right)}\left(j\right)}\right)^{2}+\right.\\
 &  & \left.\left(-1\right)\left(x_{f\left(i\right)}-x_{f^{\left(2\right)}\left(i\right)}\right)^{2}+2\left(x_{f\left(j\right)}-x_{f^{\left(2\right)}\left(j\right)}\right)\left(x_{f^{\left(2\right)}\left(j\right)}-x_{j}\right)-2\left(x_{f\left(i\right)}-x_{f^{\left(2\right)}\left(i\right)}\right)\left(x_{f^{\left(2\right)}\left(i\right)}-x_{i}\right)\right)^{1-k_{ij}}.
\end{array}
\]
Let 
\[
\text{GrL}\left(G_{f}\right)\,:=\left\{ G_{\sigma f\sigma^{\left(-1\right)}}:\sigma\in\nicefrac{\text{S}_{n}}{\text{Aut}\left(G_{f}\right)}\;\text{ and }\;n=\left|\left\{ \left|\sigma f\sigma^{\left(-1\right)}\left(j\right)-j\right|\,:\,j\in\mathbb{Z}_{n}\right\} \right|\right\} ,
\]
then summing over a conjugation orbit of arbitrarily chosen coset
representatives, (where we are careful to select only one coset representative
per left coset of Aut$\left(G_{f^{\left(2\right)}}\right)$), yields
the congruence identity
\[
\left(\sum_{\sigma\in\nicefrac{\text{S}_{n}}{\text{Aut}G_{f^{\left(2\right)}}}}P_{\sigma f\sigma^{\left(-1\right)}}\left(\mathbf{x}\right)\right)\equiv\left|\text{GrL}\left(G_{f^{\left(2\right)}}\right)\right|\prod_{0\le i\ne j<n}\left(j-i\right)\left(j^{2}-i^{2}\right)+
\]
\[
\left(-1\right)^{{n \choose 2}}\prod_{0\le i\ne j<n}\left(j-i\right)\sum_{\sigma\in\nicefrac{\text{S}_{n}}{\text{Aut}G_{f^{\left(2\right)}}}}\sum_{\begin{array}{c}
k_{ij}\in\left\{ 0,1\right\} \\
0=\underset{i\ne j}{\prod}k_{ij}
\end{array}}\prod_{0\le i\ne j<n}\left(\left(x_{\sigma f^{\left(2\right)}\sigma^{\left(-1\right)}\left(j\right)}-x_{j}\right)^{2}-\left(x_{\sigma f^{\left(2\right)}\sigma^{\left(-1\right)}\left(i\right)}-x_{i}\right)^{2}\right)^{k_{ij}}\times
\]
\[
\left(2\left(x_{\sigma f\sigma^{\left(-1\right)}\left(j\right)}-x_{\sigma f^{\left(2\right)}\sigma^{\left(-1\right)}\left(j\right)}\right)\left(x_{\sigma f^{\left(2\right)}\sigma^{\left(-1\right)}\left(j\right)}-x_{j}\right)-2\left(x_{\sigma f\sigma^{\left(-1\right)}\left(i\right)}-x_{\sigma f^{\left(2\right)}\sigma^{\left(-1\right)}\left(i\right)}\right)\left(x_{\sigma f^{\left(2\right)}\sigma^{\left(-1\right)}\left(i\right)}-x_{i}\right)+\right.
\]
\[
\left.\left(x_{\sigma f\sigma^{\left(-1\right)}\left(j\right)}-x_{\sigma f^{\left(2\right)}\sigma^{\left(-1\right)}\left(j\right)}\right)^{2}-\left(x_{\sigma f\sigma^{\left(-1\right)}\left(i\right)}-x_{\sigma f^{\left(2\right)}\sigma^{\left(-1\right)}\left(i\right)}\right)^{2}\right)^{1-k_{ij}}\text{mod}\left\{ \begin{array}{c}
\underset{i\in\mathbb{Z}_{n}}{\sum}\left(\left(x_{i}\right)^{k}-i^{k}\right)\\
0<k\le n
\end{array}\right\} .
\]
 Since $\text{Aut}\left(G_{f}\right)\subsetneq\text{Aut}\left(G_{f^{\left(2\right)}}\right)$,
it follows that the polynomial
\[
\sum_{\sigma\in\nicefrac{\text{S}_{n}}{\text{Aut}G_{f^{\left(2\right)}}}}\sum_{\begin{array}{c}
k_{ij}\in\left\{ 0,1\right\} \\
0=\underset{i\ne j}{\prod}k_{ij}
\end{array}}\prod_{0\le i\ne j<n}\left(\left(x_{\sigma f^{\left(2\right)}\sigma^{\left(-1\right)}\left(j\right)}-x_{j}\right)^{2}-\left(x_{\sigma f^{\left(2\right)}\sigma^{\left(-1\right)}\left(i\right)}-x_{i}\right)^{2}\right)^{k_{ij}}\times
\]
\[
\left(2\left(x_{\sigma f\sigma^{\left(-1\right)}\left(j\right)}-x_{\sigma f^{\left(2\right)}\sigma^{\left(-1\right)}\left(j\right)}\right)\left(x_{\sigma f^{\left(2\right)}\sigma^{\left(-1\right)}\left(j\right)}-x_{j}\right)-2\left(x_{\sigma f\sigma^{\left(-1\right)}\left(i\right)}-x_{\sigma f^{\left(2\right)}\sigma^{\left(-1\right)}\left(i\right)}\right)\left(x_{\sigma f^{\left(2\right)}\sigma^{\left(-1\right)}\left(i\right)}-x_{i}\right)+\right.
\]
\[
\left.\left(x_{\sigma f\sigma^{\left(-1\right)}\left(j\right)}-x_{\sigma f^{\left(2\right)}\sigma^{\left(-1\right)}\left(j\right)}\right)^{2}-\left(x_{\sigma f\sigma^{\left(-1\right)}\left(i\right)}-x_{\sigma f^{\left(2\right)}\sigma^{\left(-1\right)}\left(i\right)}\right)^{2}\right)^{1-k_{ij}}.
\]
does not lie in the ring of symmetric polynomial in the entries of
$\mathbf{x}$. Consequently there must be at least one evaluation
point $\mathbf{x}$ subject to the constraints
\[
\left\{ 0=\sum_{i\in\mathbb{Z}_{n}}\left(\left(x_{i}\right)^{k}-i^{k}\right),\quad\forall\:0<k\le n\right\} ,
\]
for which the evaluation of the said polynomial depend on the choice
of coset representatives. For otherwise, evaluations of the said polynomial
would be independent of choices of coset representatives. In which
case the polynomial obtained by summing over the conjugation orbit
of coset representatives would be symmetric and thereby contradict
the premise $\text{Aut}\left(G_{f}\right)\subsetneq\text{Aut}\left(G_{f^{\left(2\right)}}\right)$.
In fact, by the fundamental theorem of symmetric polynomials, we know
that
\[
\left(\sum_{\sigma\in\nicefrac{\text{S}_{n}}{\text{Aut}G_{f^{\left(2\right)}}}}P_{\sigma f\sigma^{\left(-1\right)}}\left(\mathbf{x}\right)\right)\equiv\left|\text{GrL}\left(G_{f}\right)\right|\prod_{0\le i\ne j<n}\left(j-i\right)\left(j^{2}-i^{2}\right)+
\]
\[
\sum_{\sigma\in\nicefrac{\text{S}_{n}}{\text{Aut}G_{f^{\left(2\right)}}}}\left(\sum_{\gamma\in\left(\nicefrac{\text{Aut}G_{f^{\left(2\right)}}}{\text{Aut}G_{f}}\right)\backslash\left\{ \text{id}\right\} }P_{\sigma\gamma f\left(\sigma\gamma\right)^{\left(-1\right)}}\left(\mathbf{x}\right)\right)\mod\left\{ \begin{array}{c}
\underset{i\in\mathbb{Z}_{n}}{\sum}\left(\left(x_{i}\right)^{k}-i^{k}\right)\\
0<k\le n
\end{array}\right\} .
\]
We thus conclude that
\[
P_{f}\left(\mathbf{x}\right)\not\equiv0\text{ mod}\left\{ \begin{array}{c}
\left(x_{k}\right)^{\underline{n}}\\
k\in\mathbb{Z}_{n}
\end{array}\right\} .
\]
\end{proof}
As pointed out by in \cite{gnang2020graceful}, Lem. (5) yields as
an immediate corollary that all trees are graceful. When our proposed
variant is taken in conjunction with the result of Abrham and Kotzig
\cite{ABRHAM19941}, Lem. (5) also establishes that the graphs $C_{1}\cup(2^{s}\,$C$_{2^{t}})$
(consisting of $2^{s}$ copies of a $2^{t}$-cycle and a trivial component
made up of a vertex with an attached loop edge) where $s>0$ and $t>1$
are graceful.

\section*{Acknowledgement}

This paper is based upon work supported by the Defense Technical Information
Center (DTIC) under award number FA8075-18-D-0001/0015. The views
and conclusions contained in this work are those of the authors and
should not be interpreted as necessarily representing the official
policies, either expressed or implied, of the DTIC.

\bibliographystyle{amsalpha}
\bibliography{A_new_proof_of_the_composition_Lemma}

\end{document}